\documentclass[12pt]{amsart}
\usepackage[latin9]{inputenc}
\usepackage{mathrsfs}
\usepackage{amstext}
\usepackage{amsthm}
\usepackage{amssymb}
\usepackage{wasysym}
\usepackage{esint}

\makeatletter
\numberwithin{equation}{section}
\numberwithin{figure}{section}
\theoremstyle{plain}
\newtheorem{thm}{\protect\theoremname}
  \theoremstyle{plain}
  
  \theoremstyle{plain}
  
  \theoremstyle{plain}
  \newtheorem{lemma}[thm]{Lemma}
  \newtheorem{remark}[thm]{Remark}

\@ifundefined{date}{}{\date{}}
\usepackage[colorlinks=true, bookmarks=true,pdfstartview=FitV, 
linkcolor=blue, citecolor=blue, urlcolor=blue]{hyperref}

\usepackage{upref}\usepackage{enumerate}\usepackage{tikz}
\usepackage{amscd}\usepackage{amsxtra}\usepackage{latexsym}
\usepackage{yfonts}
\usepackage{mathrsfs}\usepackage{bbm}
\usepackage{bm}
\usepackage{eso-pic}
\usepackage{graphicx}

\def\Xint#1{\mathchoice
{\XXint\displaystyle\textstyle{#1}}%
{\XXint\textstyle\scriptstyle{#1}}%
{\XXint\scriptstyle\scriptscriptstyle{#1}}%
{\XXint\scriptscriptstyle\scriptscriptstyle{#1}}%
\!\int}
\def\XXint#1#2#3{{\setbox0=\hbox{$#1{#2#3}{\int}$}
\vcenter{\hbox{$#2#3$}}\kern-.5\wd0}}

\def\avgint{\Xint-}

\calclayout

\usepackage{pdfsync}

\def\esss{\operatornamewithlimits{ess\,sup}}

\numberwithin{equation}{section}
\allowdisplaybreaks

\newcommand{\ra}{\rightarrow}

\newcommand{\bey}{\begin{eqnarray*}}
\newcommand{\eey}{\end{eqnarray*}}
\newcommand{\ba}{\begin{align}}
\newcommand{\ea}{\end{align}}
\newcommand{\bea}{\begin{align*}}
\newcommand{\ena}{\end{align*}}
\newcommand{\be}{\begin{equation}}
\newcommand{\ee}{\end{equation}}
\newcommand{\R}{\mathbb R}

\newcommand{\N}{\mathbb N}

\newcommand{\bc}{\begin{center}}
\newcommand{\ec}{\end{center}}

\newcommand{\al}{\alpha}

\providecommand{\theoremname}{Theorem}

\makeatother

  \providecommand{\corollaryname}{Corollary}
  \providecommand{\propositionname}{Proposition}
\providecommand{\theoremname}{Theorem}

\begin{document}

\author{Adam Mair }

\address{Adam Mair \\
 Department of Mathematics \\
 University of Alabama \\
 Tuscaloosa, AL 35487, USA}

\email{acmair@crimson.ua.edu}

\author{Kabe Moen}

\address{Kabe Moen \\
 Department of Mathematics \\
 University of Alabama \\
 Tuscaloosa, AL 35487, USA}

\email{kabe.moen@ua.edu}

\author{Yongming Wen}

\address{Yongming Wen, School of Mathematics and Statistics, Minnan Normal University, Zhangzhou 363000,  China} 

\email{wenyongmingxmu@163.com}

\title[Bump conditions for general commutators]{Bump conditions for general iterated commutators with applications to compactness}

\maketitle

\begin{abstract}
We prove new sufficient bump conditions for general iterated commutators of Calder\'on-Zygmund operators and fractional integral operators.  As an application of our results we derive two weight compactness theorems for higher order iterated commutators with a $CMO$ function.
\end{abstract}

\section{Introduction}

Let $L$ be a linear integral operator and consider the commutator of $L$ and a function $b$ is given by
$$[b,L]f(x)=b(x)Lf(x)-L(bf)(x).$$
In this paper we consider the general iterated commutator with symbol ${\mathbf b}=(b_1,\ldots,b_m)$
$$L_{\mathbf b}f(x)=[b_m[\ldots [b_1,L]]\ldots]f(x)$$
when $L$ is a Calder\'on-Zygmund operator $T$ or fractional integral operator $I_\al$. The action 
$$(b_1,\ldots,b_m,f) \mapsto L_{\mathbf b}f$$
defines an $(m+1)$-linear operator. For this reason, $L_{\mathbf b}$ is sometimes referred to as a multilinear commutator; not to be confused with a commutator of a multilinear operator.   In the case of a Calder\'on-Zygmund operator with kernel $K(x,y)$ it is formally given by
$$T_{\mathbf b}f(x)=\int_{\R^n}\prod_{j=1}^m(b_j(x)-b_j(y)) K(x,y)f(y)\,dy.$$
These general iterated commutators where first introduced by P\'erez and Trujillo-Gonz\'alez \cite{MR1895740}, who proved endpoint estimates and weighted inequalities. When $L$ is the fractional integral operator $I_\al$ 
$$I_{\al,\mathbf b}f(x)=\int_{\R^n}\prod_{j=1}^m(b_j(x)-b_j(y)) \frac{f(y)}{|x-y|^{n-\al}}\,dy.$$
 Later, Gorosito, Pradolini, and Salinas \cite{MR2352296} proved similar results for $I_{\al,\mathbf b}$.  The operators $L_{(b,\ldots,b)}=L^m_{b}$ are called the $m$-th iterated commutators and have been studied by a number of mathematicians.

It is well known that the $A_p$ condition
\[ [w]_{A_p}=\sup_Q \left( \fint_Q w \right) \left( \fint_Q w^{1-p'} \right)^{p-1} < \infty \]
is sufficient for singular integral operators to be bounded on $L^p(w)$.  However Muckenhoupt and Wheeden in \cite{MR0417671} showed that the correlating two-weight $A_p$ condition
\[ \sup_Q \left( \fint_Q u \right)^{\frac1p} \left( \fint_Q v^{1-p'} \right)^{\frac1{p'}}=\sup_Q\|u^{\frac1p}\|_{L^p(\frac{Q}{|Q|})}\|v^{-\frac1p}\|_{L^{p'}(\frac{Q}{|Q|})}  < \infty \]
is almost never sufficient for the $L^p(v)\rightarrow L^p(u)$ two weight boundedness of operators (see \cite{MR2797562} for more details).  To ameliorate this, the stronger $A_p$-bump condition was introduced:
\[ \sup_Q \|u^{\frac{1}{p}}\|_{X,Q}\|v^{-\frac{1}{p}}\|_{Y,Q} < \infty, \]
where $X$ and $Y$ are slightly larger norms than $L^p$ and $L^{p'}$ respectively. To state these results, we recall some background material on Young functions and Orlicz spaces. A function $A:[0,\infty)\ra [0,\infty)$ is called a Young function if it is increasing, convex, $A(0) = 0$ and $A(t)/t \rightarrow \infty$ as $t\ra \infty$. Given a Young function $A$, there exists another Young function $\bar{A}$ which we call the associate function of $A$, and the two functions satisfy
\[ A^{-1}(t) \bar{A}^{-1}(t) \approx t. \]
The Orlicz average with respect to $A$ of $f$ over $Q$ is given by
\[ \|f\|_{A,Q} = \inf \left\{ \lambda > 0 \,:\, \fint_Q A \left( \frac{|f(x)|}{\lambda} \right)dx \leq 1 \right\}. \]
The quantity $\|\cdot\|_{A,Q}$ is indeed a norm.  Also note that if $A(t) = t^p$, $p \geq 1$, then
\[ \|f\|_{A,Q} = \|f\|_{L^p,Q}=|Q|^{-\frac1p}\|f\chi_Q\|_{L^p} \]
is the usual $L^p$ average. When $A(t)=t^p\log(e+t)^a$ we will write
\[ \|f\|_{A,Q} =\|f\|_{L^p(\log L)^a,Q} . \]
For $B(t)=e^t-1$ we will write
\[ \|f\|_{B,Q} =\|f\|_{\exp L,Q} . \]

The $B_p$ integrability condition was introduced by P\'erez in \cite{MR1291534} as a way to quantify the bump conditions.  Given $1<p<\infty$, we say that a Young function belongs to $B_p$ if
\begin{equation}\label{Bp}\int_1^\infty \frac{A(t)}{t^p}\,\frac{dt}{t}<\infty.\end{equation}
Typical Young functions that satisfy the $B_p$ condition are $A(t)=t^{p-\delta}$ and $A(t)=t^p\log(e+t)^{-1-\delta}$ for some $\delta>0$. P\'erez showed that the condition
$$\sup_Q\left( \fint_Q u \right)^{\frac1p} \|v^{-\frac1p}\|_{B,Q}<\infty$$
for some $\bar{B}\in B_p$ is sufficient for the Hardy-Littlewood maximal function
$$Mf(x)=\sup_{Q\ni x}\fint_Q |f(y)|\,dy,$$
to be bounded from $L^p(v)$ to $L^p(u)$.  P\'erez also showed that the condition
\begin{equation}\label{doublebumpfrac}\sup_Q|Q|^{\frac{\al}{n}+\frac1q-\frac1p}\|u^{\frac1q}\|_{A,Q} \|v^{-\frac1p}\|_{B,Q}<\infty\end{equation}
where $\bar{A}\in B_{q'}$ and $\bar{B}\in B_p$ is sufficient for $I_\al:L^p(v)\ra L^q(u)$ when $1<p\leq q<\infty$.  In the strict off-diagonal case, $p<q$, Cruz-Uribe and the second author improved \eqref{doublebumpfrac} in two ways. First, the  $B_{p,q}$ condition is introduced for a Young function $A$
\begin{equation}\label{Bpq}\int_1^\infty \frac{A(t)^{\frac{q}{p}}}{t^q}\,\frac{dt}{t}<\infty.\end{equation}
The $B_{p,q}$ is a weaker condition on the Young function in the sense that $B_p\subsetneq B_{p,q}$ when $p<q$ (of course $B_p=B_{p,p}$). Second, they separated the bumps and showed that
$$\sup_Q|Q|^{\frac{\al}{n}+\frac1q-\frac1p}\left( \fint_Q u \right)^{\frac1p} \|v^{-\frac1p}\|_{A,Q}+\sup_Q|Q|^{\frac{\al}{n}+\frac1q-\frac1p}\|u^{\frac1q}\|_{B,Q}\left( \fint_Q v^{1-p'} \right)^{\frac1{p'}}<\infty$$
for some $\bar{A} \in B_{q',p'}$ and $\bar{B}\in B_{p,q}$ is sufficient for $I_\al:L^p(v)\ra L^q(u)$.  Interestingly, these separated bumps are not know to be sufficient in the case $p=q$ for $I_\al$. For Calder\'on-Zygmund operators Lerner \cite{Lerner2weightbound} showed that 
\begin{equation}\label{doublebumpfrac}\sup_Q\|u^{\frac1p}\|_{A,Q} \|v^{-\frac1p}\|_{B,Q}<\infty\end{equation}
where $\bar{A}\in B_{p'}$ and $\bar{B}\in B_p$ is sufficient for $T$ and for the maximal truncation operator $T^\sharp$ to be bounded from $L^p(v)$ to $L^p(u)$.

For commutators, the story is more complicated.  Cruz-Uribe and the second author \cite{DCUKabe_SharpInequalities_Commutators} began the study of bump conditions for commutators with $BMO$ functions (see Section.  
In \cite{DCUKabe_SharpInequalities_Commutators} it shown that the condition
\begin{equation}\label{commbump}\sup_Q\|u^{\frac1p}\|_{L^p(\log L)^{2p-1+\delta},Q} \|v^{-\frac1p}\|_{L^{p'}(\log L)^{2p'-1+\delta},Q}<\infty\end{equation}
for some $\delta>0$ is sufficient for the inequality $$\|[b,T]f\|_{L^p(u)}\leq C\|b\|_{BMO}\|f\|_{L^p(v)}.$$  
For fractional integrals  
\begin{equation}\label{commbumpfrac}\sup_Q|Q|^{\frac{\al}{n}+\frac1q-\frac1p}\|u^{\frac1q}\|_{L^q(\log L)^{(m+1)q-1+\delta},Q} \|v^{-\frac1p}\|_{L^{p'}(\log L)^{(m+1)p'-1+\delta},Q}<\infty\end{equation}
for $\delta>0$ is sufficient for 
  $$\|(I_\al)_b^mf\|_{L^p(u)}\leq C\|b\|^m_{BMO}\|f\|_{L^p(v)}.$$  
(see also \cite{MR2250646}).  Recently these inequalities have been improved in several directions. First Lerner, Ombrosi, and Rivera-R\'ios \cite{Iterated_Comm_Bound} found a way to separate the bump conditions that works for higher order commutators of Calder\'on-Zygmund operators $T_b^m$.  However, Isralowitz, Treil, and Pott \cite{MR4454483} and Cruz-Uribe, the second author, and Tran \cite{QuanOscClass} found a way to combine the oscillation class with the weighted class.  In particular they showed that
\begin{equation}\label{oscbump}\sup_{Q}\|(b-b_Q)^mu^{\frac1p}\|_{A,Q}\|v^{-\frac1p}||_{B,Q}+\sup_Q\|u^{\frac1p}\|_{A,Q}\|(b-b_Q)^mv^{-\frac1p}||_{B,Q}<\infty\end{equation}
for some $\bar{A}\in B_{p'}$ and $\bar{B}\in B_p$ is sufficient for 
 $$\|T^m_b f\|_{L^p(u)}\leq C\|f\|_{L^p(v)}.$$
 The last author and Wu \cite{fracbumps} extended this theory to the fractional integral case, showing that
 \begin{multline}\label{oscbumpfrac}\sup_{Q}|Q|^{\frac{\al}{n}+\frac1q-\frac1p}\|(b-b_Q)^mu^{\frac1q}\|_{A,Q}\|v^{-\frac1p}||_{B,Q}\\+\sup_Q|Q|^{\frac{\al}{n}+\frac1q-\frac1p}\|u^{\frac1q}\|_{A,Q}\|(b-b_Q)^mv^{-\frac1p}||_{B,Q}<\infty\end{multline}
is sufficient for $(I_\al)^m_b:L^p(v)\ra L^q(u)$ when $\bar{A}\in B_{q'}$ and $\bar{B}\in B_{p,q}$.  The oscillation condition \eqref{oscbump} is related to Bloom's $BMO$, in the sense that it is a weighted $BMO$ assumption. However, it does not require any strong assumptions on the individual weights $u$ and $v$.  Moreover, it is important because the condition \eqref{oscbump} implies \eqref{commbumpfrac} when $b\in BMO$ and condition \eqref{oscbump} can be used to study $b$ in different oscillation classes, for example $b^k\in BMO$ for $k\in \N$ (see \cite{QuanOscClass}).  
 
In this article we extend the oscillation bump conditions \eqref{oscbump} and \eqref{oscbumpfrac} to the operators $T_{\mathbf b}$ and $I_{\al,\mathbf b}$. This is not done to generalize for the sake of generalizing, but rather as a necessary tool to study the compactness of the higher order iterated commutators.  Our main theorem for the general iterated commutator is the following.
 
\begin{thm}\label{Norm_ineq_for_CZO}
    Let $1 < p  < \infty$, $\mathbf{b} = (b_1, \ldots, b_m) \in L_{\normalfont{loc}}^1(\R^n)^m$, and $T$ be a Calder\'{o}n-Zygmund operator.  If $(u,v)$ are weights such that
    \begin{equation*}\label{Norm_ineq_sup_bound}
        \mathsf{K}=\sum_{\tau\subseteq \{1,\ldots,m\}}\sup_Q
        \left\| \prod_{i \in \tau} (b_i - (b_i)_Q)u^{\frac{1}{p}} \right\|_{A,Q} \left\| \prod_{l \in \tau^c} (b_l - (b_l)_Q)v^{-\frac{1}{p}} \right\|_{B,Q}<\infty 
    \end{equation*}
    for $\bar{A}\in B_{p'}$ and $\bar{B}\in B_p$, then 
    $$\|T_{\mathbf{b}}f\|_{L^p(u)} \lesssim \mathsf{K} \|f\|_{L^p(v)}.$$  
\end{thm}
For the fractional integral operator we work in the lower triangular case $p\leq q$.  
 \begin{thm}\label{Norm_ineq_for_FracOp}
    Let $1 < p\leq q  < \infty$, $0<\al<n$, and $\mathbf{b} = (b_1, \ldots, b_m) \in L_{\normalfont{loc}}^1(\R^n)^m$.  If $(u,v)$ are weights such that
    \begin{equation*}\label{Norm_ineq_sup_bound}
        \mathsf{K_\al}=\sum_{\tau\subseteq \{1,\ldots,m\}}\sup_Q|Q|^{\frac{\al}{n}+\frac1q-\frac1p}
        \left\| \prod_{i \in \tau} (b_i - (b_i)_Q)u^{\frac{1}{q}} \right\|_{A,Q} \left\| \prod_{l \in \tau^c} (b_l - (b_l)_Q)v^{-\frac{1}{p}} \right\|_{B,Q}<\infty 
    \end{equation*}
    for any $p\leq s\leq q$ such that $\bar{A}\in B_{q',s'}$ and $\bar{B}\in B_{p,s}$, then 
    $$\|I_{\al,\mathbf{b}}f\|_{L^p(u)} \lesssim \mathsf{K}_\al \|f\|_{L^p(v)}.$$  
\end{thm}
\begin{remark} When $p<q$ we have a scale of conditions corresponding to $s$ with $p\leq s\leq q$.  When $s=p$ we have $\bar{A}\in B_{q',p'}$ and $\bar{B}\in B_{p}$ and when $s=q$ we have $\bar{A}\in B_{q'}$ and $\bar{B}\in B_{p,q}$.  This scale of conditions was first found by Tran \cite{QuanPhD}
\end{remark}

When we turn to the bump conditions on the individual weights we have the following Corollaries. 

\begin{thm}\label{Norm_ineq_for_CZO_BMO}
    Let $1 < p  < \infty$, $\mathbf{b} = (b_1, \ldots, b_m) \in BMO^m$, and $T$ be a Calder\'{o}n-Zygmund operator.  If $(u,v)$ are weights such that
    \begin{equation*}
        \hat{\mathsf K}= \sup_Q\|u^{\frac{1}{p}} \|_{L^p(\log L)^{(m+1)p-1+\delta},Q} \| v^{-\frac{1}{p}} \|_{L^{p'}(\log L)^{(m+1)p'-1+\delta},Q}<\infty 
    \end{equation*}
    for some $\delta>0$, then 
    $$\|T_{\mathbf{b}}f\|_{L^p(u)} \lesssim\hat{\mathsf K}\prod_{j=1}^m\|b_j\|_{BMO} \|f\|_{L^p(v)}.$$  
\end{thm}
 \begin{thm}\label{Norm_ineq_for_FracOp_BMO}
    Let $1 < p\leq q  < \infty$, $0<\al<n$, and $\mathbf{b} = (b_1, \ldots, b_m) \in BMO^m$.  If $(u,v)$ are weights such that
        \begin{equation*}
       \hat{\mathsf K}_\al= \sup_Q|Q|^{\frac\al{n}+\frac1q-\frac1p}
        \|u^{\frac{1}{q}} \|_{L^q(\log L)^{(m+\frac1{s'})q+\delta},Q} \| v^{-\frac{1}{p}} \|_{L^{p'}(\log L)^{(m+\frac1s)p'+\delta},Q}<\infty 
    \end{equation*}
    for any $p\leq s\leq q$ and $\delta>0$ then 
    $$\|I_{\al,\mathbf{b}}f\|_{L^p(u)} \lesssim  \hat{\mathsf K}_\al\prod_{j=1}^m\|b_j\|_{BMO} \|f\|_{L^p(v)}.$$  
\end{thm}
\begin{remark} For $p\leq s\leq q$, the powers on logarithmic terms satisfy $(m+\frac1{s'})q+\delta\leq (m+1)q-1+\delta$ and $(m+\frac1{s})p'+\delta\leq (m+1)p'-1+\delta$ which were the previous powers by Li \cite{MR2250646}.  However, when $p=q$ the condition collapses to
 \begin{equation*}
        \sup_Q|Q|^{\frac\al{n}}
        \|u^{\frac{1}{p}} \|_{L^p(\log L)^{(m+1)p+\delta},Q} \| v^{-\frac{1}{p}} \|_{L^{p'}(\log L)^{(m+1)p'+\delta},Q}<\infty .
         \end{equation*}

\end{remark}


The main application our results is to prove compactness of higher order iterated commutators. Recall that a linear operator $T:X \rightarrow Y$ between two Banach spaces is compact if $T(B_X)$ has compact closure in $Y$ ($B_X$ being the the unit ball in $X$). In \cite{MR467384} Uchiyama showed that restricting $b$ to a subset $CMO \subseteq BMO$ gives us sufficient conditions for a Calder\'on-Zygmund operator to satisfy $[b,T]:L^p \rightarrow L^p$ being a compact operator (see \cite{2022arXiv220810311M} for more context on recent work regarding compactness with respect to $T$).  The first and second authors in \cite{2022arXiv220810311M} established conditions on a two-weight system $(u,v)$ such that $b \in CMO$ is sufficient to establish that $[b,T]:L^p(v) \rightarrow L^q(u)$ is a compact operator.  When studying the compactness of the operator $[b,T]=T^1_b$ the linearity of the map $b\mapsto [b,T]$ is important and does not readily transfer to the operator $T^m_b$.  However, by analyzing the operator $T_{\mathbf b}$ we are able to avoid this obstacle. As an application to the general iterated norm inequalities we prove in section \ref{sparse_section}, we extend this result to iterated commutators:

\begin{thm}\label{main_app_CZO}
    Let $1 < p < \infty$, $b \in CMO(\R^n)$, and $T$ a Calder\'{o}n-Zygmund operator.  If $(u,v)$ are a pair of weights satisfying
    \[ \sup_Q \|u^{\frac{1}{p}}\|_{L^p(\log L)^{(m+1)p-1+\delta},Q}\|v^{-\frac{1}{p}}\|_{L^{p'}(\log L)^{(m+1)p'-1+\delta},Q} < \infty \]
    for some $\delta > 0$, then $T_b^m$ is a compact operator from $L^p(v)$ to $L^p(u)$ for all natural numbers $m$.
\end{thm}

In the one weight setting it is shown in \cite{MR3834654} that $b \in CMO$ characterizes compactness of the operator $[b,I_\alpha]:L^p(w^p) \rightarrow L^q(w^q)$.  However this requires the assumption that $w \in A_{p,q}$, which implies that $w^p \in A_p$ and $w^q \in A_q$.  We will prove general iterated commutator norm inequalities in section \ref{sparse_section} using similar sparse domination arguments to the ones used for Calder\'{o}n-Zygmund operators.  We then prove bump conditions for iterated commutators of the Riesz potential to be compact operators.

\begin{thm}
    \label{main_app_frac}
    Let $0 < \alpha < n$, $1 < p \leq q < \infty$, $b \in CMO(\R^n)$. If $(u,v)$ are a pair of weights satisfying 
    \begin{equation}\label{Iterated_Riesz_Commutator_Bump_Conditions}
        \sup_Q |Q|^{\frac{\alpha}{n} + \frac{1}{q} - \frac{1}{p}}\|u^{\frac{1}{q}}\|_{L^q(\log L)^{(m+\frac1{s'})q+\delta},Q}\|v^{-\frac{1}{p}}\|_{L^{p'}(\log L)^{(m+\frac1s)p' +\delta},Q} < \infty
    \end{equation}
    for any $p\leq s\leq q$ and $\delta > 0$, then $(I_\alpha)_b^m:L^p(v) \rightarrow L^q(u)$ is a compact operator for all natural numbers $m$.
\end{thm}

\section{Preliminaries}\label{prelim_section}

Given a measurable function $b$ and a cube $Q$, by $b_Q$ we denote the average value of $b$ on $Q$:
\[ b_Q = \fint_Q b(x)dx. \]
We say $b$ is of bounded mean oscillation, denoted $BMO$, if
\[ \|b\|_{BMO} = \sup_Q \fint_Q \left| b(x) - b_Q \right|dx < \infty. \]
This quantity fails to be a norm since any constant function $c$ is such that $\|c\|_{BMO} = 0$.  We remedy this by considering the space $BMO$ modulo constants.  Define the space $CMO$ as the closure of $C_c^\infty (\R^n)$ in the $BMO$ norm $\|\cdot\|_{BMO}$. Recall that $BMO$ functions satisfy the John-Nirenberg inequality,
$$\fint_Q\exp\Big(\frac{c|b(x)-b_Q|}{\|b\|_{BMO}}\Big)\,dx\leq C$$
for some constants $c,C>0$.  In terms of Orlicz norms the John-Nirenberg inequality implies
$$\|b\|_{BMO}\approx \sup_Q\|b-b_Q\|_{\exp L,Q}.$$

The following characterizations follow the work by Cruz-Uribe, Martell, and P\'{e}rez in \cite{MR2797562}.  When working with a pair of weights $(u,v)$ we think of non-negative, locally integrable functions such that $u$ is positive on a set of positive measure and $v$ is positive almost everywhere.  

%
%
%
%
%
%

    %
    %

We need the following multilinear version H\"older inequality for Orlicz spaces which was proven in \cite{MR1895740}.

\begin{lemma}\label{Generalized_Young_Holders}
    Let $A_1, \ldots , A_n, C$ be non-negative, continuous, strictly increasing functions on $[0,\infty)$ that satisfy
    \[ (A_1^{-1}\cdots A_n^{-1})(t) \leq C^{-1}(t) \text{ for all } t \geq 0. \]
    Also assume $C$ is a Young function.  If $Q$ is a cube and $f_1, \ldots , f_n$ are measurable functions, then
    \[ \|f_1 \cdots f_n\|_{C,Q} \leq n\|f_1\|_{A_1,Q} \cdots \|f_n\|_{A_n,Q}. \]
\end{lemma}

%
%

    %
    %
    %

Given a Young function $A$, define the maximal operator associated with $A$ by
\[ M_A f(x) = \sup_{Q \ni x} \|f\|_{A,Q}. \]
%
%
%
P\'{e}rez showed in \cite{CZTheory} that the $B_p$ integrability condition, \eqref{Bp}, characterizes the $L^p(\R^n)$ boundedness of $M_A$, namely, $A \in B_p$ if, and only if, $M_A:L^p(\R^n) \rightarrow L^p(\R^n)$.  We will use condition introduced in \cite{CM} similar to the $B_p$ condition but better suited to our fractional bump conditions.  If we define the fractional Orlicz maximal operator
\[ M_{\beta,A}f(x) = \sup_{Q \ni x} |Q|^{\beta/n}\|f\|_{A,Q}. \]
If $\frac{\beta}{n} = \frac{1}{p} - \frac{1}{q}$ the authors in \cite{CM} prove that the $B_{p,q}$ condition on $A$, \eqref{Bpq}, is sufficient for the boundedness $M_{\beta,A}:L^p(\R^n) \rightarrow L^q(\R^n)$.

Recall that $T$ is a Calder\'{o}n-Zygmund operator (CZO) on $\mathbb{R}^n$ if $T$ is bounded on $L^2(\mathbb{R}^n)$ and it admits the following representation
\[
Tf(x)=\int_{\mathbb{R}^n}K(x,y)f(y)dy,\quad f \in L_c^\infty(\R^n), x\notin \mathsf{supp} f.
\]
The kernel $K(x,y)$ defined on $\{(x,y)\,:\,x \neq y\}$ satisfies the size condition:
\begin{align}\label{size condition}
|K(x,y)|\leq \frac{C}{|x-y|^n},~x\neq y
\end{align}
and the smoothness condition
\begin{align}\label{smooth condition}
|K(x,y)-K(x',y)|+|K(y,x)-K(y,x')|\leq\frac{C|x-x'|}{|x-y|^{n+1}},
\end{align}
for all $|x-y|>2|x-x'|$.  We will also need the maximal truncation operator which is given by
$$T^\sharp f(x)=\sup_{\eta>0}\left|\int_{|x-y|>\eta} K(x,y)f(y)\,dy\right|.$$

In \cite{MR1291534}, it is shown that if the pair of weights satisfies the condition
\begin{equation}\label{maxbump}\sup_Q\left(\avgint_Q u\right)^{\frac1p}\|v^{-\frac1p}\|_{L^{p'}(\log L)^{p'-1+\delta},Q}<\infty,\end{equation}
for some $\delta>0$ then $M:L^p(v)\ra L^p(u)$.  For CZOs bump conditions on both weights are needed. Lerner \cite{Lerner2weightbound} showed that the condition
\begin{equation}\label{czobump}\sup_Q\|u^{\frac1p}\|_{L^p(\log L)^{p-1+\delta},Q}\|v^{-\frac1p}\|_{L^{p'}(\log L)^{p'-1+\delta},Q}<\infty\end{equation}
for some $\delta>0$ is sufficient for the boundedness of $T$ and $T^\sharp$ from $L^p(v)$ to $L^p(u)$.

Define the fractional integral operator by
\[ I_\alpha f(x) = \int_{\R^n} \frac{f(y)}{|x-y|^{n-\alpha}}dy, \quad 0 < \alpha < n. \]
As mentioned in the introduction, \cite{MR1291534}, P\'{e}rez proved that if $1 < p \leq q < \infty$ and $(u,v)$ satisfy
\[ \sup_Q |Q|^{\frac{\alpha}{n}+\frac{1}{q}-\frac{1}{p}}\|u^{\frac{1}{q}}\|_{L^q(\log L)^{q-1+\delta},Q}\|v^{-\frac{1}{p}}\|_{L^{p'}(\log L)^{p'-1+\delta},Q} < \infty \]
for $\delta >0$, then $I_\alpha:L^p(v) \rightarrow L^q(u)$.  However commutators of singular integrals are more singular than their associated operators as is seen in the different conditions needed for sharp norm inequalities.  In \cite{MR2250646}, Li showed that given \eqref{Iterated_Riesz_Commutator_Bump_Conditions}, then $I_{\alpha,b}^m:L^p(v) \rightarrow L^q(u)$.  This is sharp in the sense that it is not true for $\delta =0$ (see \cite{DCUKabe_SharpInequalities_Commutators}).  

This inequality is crucial for a reduction we make in our compactness arguments.  Note how the $m=1$ case reflects the higher singularity of the commutator from the higher power of the logarithm in our Young functions.  Also note that the condition given in \eqref{Iterated_Riesz_Commutator_Bump_Conditions} is sufficient for $I_\alpha$ and $M_\alpha$ to be bounded from $L^p(v)$ to $L^q(u)$, where $M_\alpha$ is the maximal operator associated with the Riesz potental,
\[ M_\alpha f(x) = \sup_{Q \ni x} \frac{1}{|Q|^{1- \frac{\alpha}{n}}}\int_Q |f(y)|dy. \]

To prove our compactness results will use a weighted version of the Kolmogorov-Riesz theorem due to Guo and Zhao \cite{GuoZhaoKolmog}.

\begin{lemma}\label{GuoZhao}
    Let $1 \leq p < \infty$, and let $u$ be a weight.  If $\mathcal{F} \subseteq L^p(u)$ satisfies the following conditions:
    \begin{enumerate}
        \item $\mathcal{F}$ is uniformly bounded,
        \[ \sup_{f \in \mathcal{F}}\|f\|_{L^p(u)} \lesssim 1; \]

        \item $\mathcal{F}$ uniformly vanishes at infinity:
        \[ \lim_{R\ra \infty}\sup_{f \in \mathcal{F}}\|f\chi_{\{\R^n \setminus B(0,R)\}}\|_{L^p(u)}= 0; \]

        \item $\mathcal{F}$ is uniformly equicontinuous:
        \[ \lim_{h\ra 0}\sup_{f \in \mathcal{F}}\|f(\cdot + h) - f(\cdot)\|_{L^p(u)} = 0, \]
    \end{enumerate}
    then the family $\mathcal{F}$ has compact closure in $L^p(u)$.
\end{lemma}

\section{Sparse Domination}\label{sparse_section}

Recall that a dyadic grid $\mathcal{D}$ is a collection of half-open cubes $\prod_{i=1}^n [a_i,b_i)$ such that:
\begin{enumerate}
    \item Each cube has side-length $2^k$ for some integer $k$.  If $\mathcal{D}_k$ is the collection of cubes $Q \in \mathcal{D}$ of side-length $2^k$, then $\mathcal{D}_k$ partitions $\R^n$,
    
    \medskip
    
    \item for all $x \in \R^n$, there is a unique cube in each family $\mathcal{D}_k$ containing it,
    
    \medskip
    
    \item given any two distinct cubes in $\mathcal{D}$, they are either disjoint or one is contained in the other,
    
    \medskip
    
    \item for each cube $Q \in \mathcal{D}_k$, there is a unique cube in $Q_{k+1}$ containing it; this cube is denoted $\hat{Q}$ and is called the dyadic parent of $Q$,
    
    \medskip
    
    \item if $Q \in \mathcal{D}_k$ then there are $2^n$ cubes in $\mathcal{D}_{k-1}$ contained in $Q$.
\end{enumerate}
Given a cube $Q$ then $\mathcal{D}(Q)$ is the set of all dyadic cubes with respect to $Q$, i.e. the cubes obtained by repeatedly subdividing $Q$ and its descendants into $2^n$ congruent cubes.  An important sub-family of dyadic grids are sparse families.  A family $\mathcal{S}$ of cubes in $\R^n$ is sparse if there exists $0 < \alpha < 1$ such that for all $Q \in \mathcal{S}$ there is a measurable set $E_Q \subseteq Q$ such that $|E_Q| \geq \alpha |Q|$ and the collection $\{E_Q\}_{Q \in \mathcal{S}}$ is pairwise disjoint.

We wish to establish strong norm inequalities in the two-weight setting for both $T_\mathbf{b}$ and $I_{\alpha,\mathbf{b}}$ by way of sparse domination. Given a dyadic grid $\mathcal D$ and sparse family $\mathcal S\subseteq \mathcal D$ define the sparse operator
    \[ T_{\mathcal{S}, \mathbf{b}}^{\alpha,\tau}f(x) = 
    \sum_{Q \in \mathcal{S}} |Q|^{\frac{\alpha}{n}} \left( \prod_{i \in \tau} |b_i(x) - (b_i)_Q| \avgint_Q \prod_{l \in \tau^c} |b_l(y) - (b_l)_Q| |f(y)|dy \right)\chi_Q(x) \]
where $0\leq \al<n$ and $\tau\subseteq\{1,\ldots,m\}$.  When $\al=0$ we simply write $T_{\mathcal{S}, \mathbf{b}}^{0,\tau}=T_{\mathcal{S}, \mathbf{b}}^{\tau}$.

We now show that our iterated commutators can be bounded by a linear combination of the above sparse operators. For Calder\'on-Zygmund operators we have the following lemma, which is a simplified version of Proposition 2.1 in \cite{CZOSparseDom}:

\begin{lemma}\label{CZO_Commuator_Sparse_Bound}
    Let $T$ be a Calder\'{o}n-Zygmund operator and $\tau_m=\{1,\ldots,m\}$.  Given $\mathbf{b} = (b_1, \ldots , b_m)$ with $b_i \in L_{\normalfont{loc}}^1(\R^n)$, there exists $C = C(n,T)$ such that for any $f \in L_{c}^\infty(\R^n)$, there exists $3^n$ sparse families $S_j$ of dyadic cubes such that
    \[ |T_{\mathbf{b}}f(x)| \leq C \sum_{j=1}^{3^n} \sum_{\tau \subseteq \tau_m}T_{\mathcal{S}_j, \mathbf{b}}^{\tau}f(x).  \]
\end{lemma}

For $I_{\alpha,\mathbf{b}}$ there does not seem to be known pointwise sparse domination formulas, unless $b_1=b_2=\cdots=b_m$ (see \cite{MR4124126}), so we will prove one.  First we need the following lemmas.

\begin{lemma}\rm(cf. \cite{MR3695871})\label{Three Lattice Theorem}\,\
Given a dyadic lattice $\mathcal{D}$, there exist $3^n$ dyadic lattices $\mathcal{D}_1,\dots,\mathcal{D}_{3^n}$ such that
$$\{3Q:Q\in\mathcal{D}\}=\bigcup_{j=1}^{3^n}\mathcal{D}_j$$
and for each cube $Q\in\mathcal{D}$ we can find a cube $R_Q$ in each $\mathcal{D}_j$ such that $Q\subseteq R_Q$ and $3l_Q=l_{R_Q}$.
\end{lemma}

Letting $\mathcal{D}$ be a dyadic lattice, we note that for any cube $Q\subseteq\mathbb{R}^n$ we can always find a cube $Q'\in\mathcal{D}$ such that $l_Q/2<l_{Q'}\leq l_Q$ and $Q\subset3Q'$. By the above lemma, for some $j\in\{1,\dots,3^n\}$, it is easy to see that $3Q'=P\in\mathcal{D}_j$. Hence, for each cube $Q\subset\mathbb{R}^n$, we can find a cube $P\in\mathcal{D}_j$ that satisfies $Q\subset P$ and $l_P\leq3l_Q$.

For a cube $Q_0\subset\mathbb{R}^n$, define the grand maximal truncated operator $\mathcal{M}_{I_\alpha}$ and local grand maximal truncated operator $\mathcal{M}_{I_\alpha,Q_0}$ by
$$\mathcal{M}_{I_\alpha}f(x)=\sup_{Q\ni x}\esss_{\xi\in Q}|I_\alpha(f\chi_{\mathbb{R}^n\backslash3Q})(\xi)|,$$
$$\mathcal{M}_{I_\alpha,Q_0}f(x)=\sup_{x\in Q\subset Q_0}\esss_{\xi\in Q}|I_\alpha(f\chi_{3Q_0\backslash3Q})(\xi)|,$$
respectively.

\begin{lemma}\rm(cf. \cite{MR4124126})\label{grand maximal truncated operator}
Let $0<\alpha<n$. Let $Q_0\subset \mathbb{R}^n$ be a cube. The following pointwise estimates holds:
\medskip

{\rm (1)} For a.e. $x\in Q_0$,
$$
|I_\alpha(f\chi_{3Q_0})(x)|\leq\mathcal{M}_{I_\alpha,Q_0}f(x);$$
\medskip

{\rm (2)} $\mathcal{M}_{I_\alpha}$ is bounded from $L^1(\mathbb{R}^n)$ to $L^{\frac{n}{n-\alpha},\infty}(\mathbb{R}^n)$.
\end{lemma}

\begin{lemma}\label{sparse domination for fraction}
Let $0<\alpha<n$, $I_\alpha$ be fractional integral operators and $\tau_m=\{1,\cdots,m\}$. Given $\textbf{b}(x)=(b_1(x),\cdots,b_m(x))$ with $b_i\in L_{\rm{loc}}^{1}(\mathbb{R}^n)$, there exists a constant $C=C(n,\al)$ so that for any $f\in L_c^\infty(\mathbb{R}^n)$, there exists $3^n$ sparse families $\mathcal{S}_j$ of dyadic cubes such that
\begin{align*}
|I_{\alpha,\textbf{b}}f(x)|&\leq C\sum_{j=1}^{3^n}\sum_{\tau\subseteq\{1,\ldots,m\}}T_{\mathcal{S}_j,\textbf{b}}^{\alpha, \tau}(f)(x).
\end{align*}
\end{lemma}
\begin{proof}
As we previously noted there are $3^n$ dyadic lattices such that for any cube $Q\subset \mathbb{R}^n$, there is a cube $R_Q\in\mathcal{D}_j$ for some $j$, for which $3Q\subset R_Q$ and $|R_Q|\leq9^n|Q|$.

Following a similar scheme as in \cite{MR3695871}, it reduces to show that for any cube $Q_0\subset\mathbb{R}^n$, there is a 1/2-sparse family $\mathcal{S}\subset \mathcal{D}(Q_0)$ such that for a.e. $x\in Q_0$
\begin{align}\label{claim1}
|I_{\alpha,\textbf{b}}(f\chi_{3Q_0})(x)|&\leq C\sum_{Q\in\mathcal{S}}\Big[\sum_{\tau\subset\tau_m}
\Big(\prod_{i\in\tau}|b_i(x)-( b_i)_{R_Q}|\Big)
\Big(\Big(\prod_{k\in\tau_m\backslash\tau}|b_k-( b_k)_{R_Q}||f|\Big)_{3Q}\Big)\Big]\\
&\qquad\times|3Q|^{\alpha/n}\chi_Q(x).\nonumber
\end{align}

To prove \eqref{claim1}, it suffices to prove the following recursive estimate: there is a disjoint family of cubes $P_j\in\mathcal{D}(Q_0)$ such that $\sum_j|P_j|\leq|Q_0|/2$ and for a.e. $x\in Q_0$,
\begin{align}\label{claim2}
|I_{\alpha,\textbf{b}}(f\chi_{3Q_0})(x)|&\leq C\Big[\sum_{\tau\subseteq\tau_m}
\Big(\prod_{i\in\tau}|b_i(x)-( b_i)_{R_{Q_0}}|\Big)
\Big(\Big(\prod_{k\in\tau_m\backslash\tau}|b_k-( b_k)_{R_{Q_0}}||f|\Big)_{3Q_0}\Big)\Big]\\
&\qquad\times|3Q_0|^{\alpha/n}\chi_{Q_0}(x)+\sum_j|I_{\alpha,\textbf{b}}(f\chi_{3P_j})(x)|\chi_{P_j}(x).
\nonumber
\end{align}
Indeed, iterating \eqref{claim2}, we get \eqref{claim1} with $\mathcal{S}=\{P_j^k\}$, where $\{P_j^0\}=\{Q_0\}$, $\{P_j^1\}=\{P_j\}$ and $\{P_j^k\}$ are the cubes obtained at the $k$-th stage of this iterative process.

For any mutually disjoint cubes $P_j\in\mathcal{D}(Q_0)$, \eqref{claim2} follows from the below estimate
\begin{align}\label{claim3}
&|I_{\alpha,\textbf{b}}(f\chi_{3Q_0})|\chi_{Q_0\backslash\cup_jP_j}+\sum_j|
I_{\alpha,\textbf{b}}(f\chi_{3Q_0})-I_{\alpha,\textbf{b}}(f\chi_{3P_j})|\chi_{P_j}\\
&\quad\leq C\Big[\sum_{\tau\subseteq\tau_m}
\Big(\prod_{i\in\tau}|b_i(x)-( b_i)_{R_{Q_0}}|\Big)
\Big(\Big(\prod_{k\in\tau_m\backslash\tau}|b_k-( b_k)_{R_{Q_0}}||f|\Big)_{3Q_0}\Big)\Big]|3Q_0|^{\alpha/n}\chi_{Q_0}(x).\nonumber
\end{align}

Now we are in the position to prove \eqref{claim3}. Since $I_{\alpha,\mathbf{b-c}}=I_{\alpha,\mathbf{b}}$ for any vector constant $\mathbf{c}$, the left side of inequality \eqref{claim3} is bounded by
\begin{align}\label{claim4}
&\leq\sum_{\tau\subseteq\tau_m}\prod_{i\in\tau}|b_i-( b_i)_{R_{Q_0}}|\Big|I_\alpha\Big(\prod_{k\in\tau_m\backslash\tau}(b_k-( b_k)_{R_{Q_0}})f\Big)\Big|\chi_{Q_0\backslash\cup_jP_j}\\
&\qquad+\sum_j\sum_{\tau\subseteq\tau_m}\prod_{i\in\tau}|b_i-( b_i)_{R_{Q_0}}|
\Big|I_\alpha\Big(\prod_{k\in\tau_m\backslash\tau}(b_k-( b_k)_{R_{Q_0}})f\chi_{3Q_0\backslash3P_j}\Big)\Big)\Big|\chi_{P_j}.\nonumber
\end{align}
Define $E=\cup_{\tau\subseteq\tau_m}E_{\tau}$, where
\begin{align*}
E_\tau&=\Big\{x\in Q_0:\mathcal{M}_{I_\alpha,Q_0}\Big(\prod_{k\in\tau_m\backslash\tau}(b_k-( b_k)_{R_{Q_0}})f\Big)(x)\\
&\qquad>
C|3Q_0|^{\alpha/n}\Big(\prod_{k\in\tau_m\backslash\tau}|b_k-( b_k)_{R_{Q_0}}||f|\Big)_{3Q_0}\Big\}.
\end{align*}
Note that
$$\mathcal{M}_{I_\alpha,Q_0}g\leq\mathcal{M}_{I_\alpha}(g\chi_{3Q_0}).$$
Then for each $\tau$, by Lemma \ref{grand maximal truncated operator}, we have
\begin{align*}
|E_\tau|&\leq\Big(\frac{c_{n,\alpha}\int_{3Q_0}\Big|\prod_{k\in\tau_m\backslash\tau}
(b_k-( b_k)_{R_{Q_0}})f\Big|}{C|3Q_0|^{\alpha/n}\Big(\prod_{k\in\tau_m\backslash\tau}
|b_k-( b_k)_{R_{Q_0}}||f|\Big)_{3Q_0}}
\Big)^{\frac{n}{n-\alpha}}\\
&=3^n\big(\frac{c_{n,\alpha}}{C}\big)^{\frac{n}{n-\alpha}}|Q_0|.
\end{align*}
Choose $C$ big enough such that $|E|\leq\sum_{\tau\subseteq\tau_m}|E_\tau|\leq|Q_0|/2^{n+2}$.

Applying the Calder\'{o}n-Zygmund to $\chi_E$ on $Q_0$ at height $h=1/2^{n+1}$, we get mutually disjoint cubes $P_j\in\mathcal{D}(Q_0)$ such that for each $j$,
$$2^{-n-1}|P_j|\leq|P_j\cap E|\leq 2^{-1}|P_j|,~|E\backslash\cup_j P_j|=0,$$
which further implies that
$$\sum_j|P_j|<|Q_0|/2,~P_j\cap E^c\neq\emptyset.$$

We return to prove \eqref{claim3}. For $x\in Q_0\backslash\cup_jP_j$, then $x\notin E_\tau$ yields that
\begin{align}\label{term1}
&\sum_{\tau\subseteq\tau_m}\prod_{i\in\tau}|b_i-( b_i)_{R_{Q_0}}|\Big|I_\alpha\Big(\prod_{k\in\tau_m\backslash\tau}(b_k-( b_k)_{R_{Q_0}})f\Big)\Big|\chi_{Q_0\backslash\cup_jP_j}\\
&\quad\leq\sum_{\tau\subseteq\tau_m}\prod_{i\in\tau}|b_i-( b_i)_{R_{Q_0}}|\mathcal{M}_{I_\alpha,Q_0}\Big(\prod_{k\in\tau_m\backslash\tau}(b_k-( b_k)_{R_{Q_0}})f\Big)\nonumber\\
&\quad\leq C\sum_{\tau\subseteq\tau_m}\prod_{i\in\tau}|b_i-( b_i)_{R_{Q_0}}||3Q_0|^{\alpha/n}
\Big(\prod_{k\in\tau_m\backslash\tau}|b_k-( b_k)_{R_{Q_0}}||f|\Big)_{3Q_0}.\nonumber
\end{align}
On the other hand, fix some $j$, using $P_j\cap E^c\neq\emptyset$, we deduce that
\begin{align*}
\mathcal{M}_{I_\alpha,Q_0}\Big(\prod_{k\in\tau_m\backslash\tau}(b_k-( b_k)_{Q_0})f\Big)
\leq C|3Q_0|^{\alpha/n}(\prod_{k\in\tau_m\backslash\tau}|b_k-( b_k)_{R_{Q_0}}||f|)_{3Q_0}
\end{align*}
It follows that
\begin{align}\label{term2}
&\sum_j\sum_{\tau\subseteq\tau_m}\prod_{i\in\tau}|b_i-( b_i)_{R_{Q_0}}|
\Big|I_\alpha\Big(\prod_{k\in\tau_m\backslash\tau}(b_k-( b_k)_{R_{Q_0}})f\chi_{3Q_0\backslash3P_j}\Big)\Big|\chi_{P_j}\\
&\quad\leq C\sum_{\tau\subseteq\tau_m}\prod_{i\in\tau}|b_i-( b_i)_{R_{Q_0}}||3Q_0|^{\alpha/n}
\Big(\prod_{k\in\tau_m\backslash\tau}|b_k-( b_k)_{R_{Q_0}}||f|\Big)_{3Q_0}.\nonumber
\end{align}
Hence, \eqref{claim4}, \eqref{term1} and \eqref{term2} show \eqref{claim3}. This completes the proof.
\end{proof}

To prove Theorems \ref{Norm_ineq_for_CZO} and \ref{Norm_ineq_for_FracOp} it suffices to bound $T_{\mathcal{S}, \mathbf{b}}^{\alpha,\tau}$ for an arbitrary sparse family $\mathcal S$. We have the following norm inequality for a general sparse operator.

\begin{thm}\label{sparse_operator_norm_ineq}
    Let $1 < p \leq q < \infty$, $\mathbf{b}(x) = (b_1(x), \ldots , b_m(x))$ with $b_i \in L_{\normalfont{loc}}^1(\R^n)$.  Let $0 \leq \alpha < n$, $\tau \subseteq \{1,\ldots,m\}$, and $\mathcal{S}$ be a sparse family of dyadic cubes.  Assume that $A$ and $B$ are Young functions that satisfy $\bar{A} \in B_{q',s'}$, $\bar{B} \in B_{p,s}$ for some $p\leq s\leq q$.  If $u$ and $v$ are a pair of weights such that
    \begin{equation}\label{sparse_operator_condition}
        \mathsf{K}_{\al,\tau}=\sup_{Q\in \mathcal{S}} |Q|^{\frac{\alpha}{n}+\frac{1}{q}-\frac{1}{p}}
        \left\| \prod_{i \in \tau} (b_i(x) - (b_i)_Q)u^{\frac{1}{q}} \right\|_{A,Q} \left\| \prod_{l \in \tau^c} (b_l(y) - (b_l)_Q)v^{-\frac{1}{p}} \right\|_{B,Q},
    \end{equation}
    then $\|T_{\mathcal{S}, \mathbf{b}}^{\alpha,\tau}f\|_{L^q(u)} \lesssim \mathsf{K}_{\al,\tau} \|f\|_{L^p(v)}$.
\end{thm}

\begin{proof}
    By duality it suffices to show that
    \[ \int_{\R^n} T_{\mathcal{S}, \mathbf{b}}^{\alpha,\tau}f(x)g(x)u(x)^{\frac{1}{q}}dx \]
    is bounded by $\|f\|_{L^p(v)}$, given $\|g\|_{L^{q'}}=1$.  Let $p\leq s\leq q$.  Using the generalized H\"{o}lder inequality for Young functions and letting $\frac{\beta}{n} = \frac{1}{q'} - \frac{1}{s'}=\frac1s-\frac1q$  and  $\frac{\gamma}{n} = \frac{1}{p} - \frac{1}{s}$ we have
    \begin{align*}
        \left| \int_{\R^n} T_{\mathcal{S}, \mathbf{b}}^{\alpha,\tau}f(x)g(x)u(x)^{\frac{1}{q}}dx \right|
        \leq{}& \sum_{Q \in \mathcal{S}} |Q|^{\frac{\alpha}{n}+1} \left\| \prod_{i \in \tau} (b_i - (b_i)_Q)u^{\frac{1}{q}} \right\|_{A,Q}\|g\|_{\bar{A},Q}\\
        {}& \quad \times \left\| \prod_{l \in \tau^c} (b_l - (b_l)_Q) v^{-\frac{1}{p}}\right\|_{B,Q} \|fv^{\frac{1}{p}}\|_{\bar{B},Q}\\
        \leq{}&  \mathsf{K}_{\al,\tau} \sum_{Q \in \mathcal{S}} |Q|^{\frac{1}{p}-\frac{1}{q}+1}\|g\|_{\bar{A},Q}\|fv^{\frac{1}{p}}\|_{\bar{B},Q}\\
           ={}&  \mathsf{K}_{\al,\tau} \sum_{Q \in \mathcal{S}} |Q|^{\frac{1}{p}-\frac1s+\frac1s-\frac{1}{q}}\|g\|_{\bar{A},Q}\|fv^{\frac{1}{p}}\|_{\bar{B},Q}|Q|\\
        \lesssim{}&  \mathsf{K}_{\al,\tau} \sum_{Q \in \mathcal{S}} |Q|^{\frac{\beta}{n}}\|g\|_{\bar{A},Q} |Q|^{\frac{\gamma}{n}}\|fv^{\frac{1}{p}}\|_{\bar{B},Q}|E_Q|\\
        \leq{}&  \mathsf{K}_{\al,\tau} \int_{\R^n} M_{\beta,\bar{A}}(g)(x) M_{\gamma,\bar{B}}(fv^{\frac{1}{p}})(x)dx\\
       \leq{}&  \mathsf{K}_{\al,\tau} \|M_{\beta,\bar{A}}(g)\|_{L^{s'}}\|M_{\gamma,\bar{B}}(fv^{\frac{1}{p}})\|_{L^s}\\
        \leq{}&  \mathsf{K}_{\al,\tau} \|M_{\beta,\bar{A}}\|_{L^{q'} \rightarrow L^{s'}}\|M_{\gamma,\bar{B}}\|_{L^p \rightarrow L^s}\|f\|_{L^p(v)}.
    \end{align*}
\end{proof}


    %
    %

Using Theorem \ref{sparse_operator_norm_ineq}, and Lemmas \ref{CZO_Commuator_Sparse_Bound} and \ref{sparse domination for fraction}, we have the desired norm inequality for $I_{\alpha,\mathbf{b}}$, and with the restrictions $\alpha = 0$ and $p=q$, we also attain our norm inequality for $T_\mathbf{b}$.

Finally, we end with the proof of Theorem \ref{Norm_ineq_for_CZO_BMO} and note that the proof of Theorem \ref{Norm_ineq_for_FracOp_BMO} is similar.

\begin{proof}[Proof of Theorem \ref{Norm_ineq_for_CZO_BMO}]
    We show that with certain $\bar{A} \in B_{p'}$ and $\bar{B} \in B_p$ and $\mathbf b\in BMO^m$ then the conditions of Theorem \ref{Norm_ineq_for_CZO} are satisfied. Let $\tau \subseteq \{1,\ldots,m\}$ with $|\tau|=j$ and let $U(t) = t^p\log(e+t)^{(m+\frac{1}{p'})p + \delta}$ for some $\delta > 0$ and $\phi(t) = e^t - 1$.  We choose the following function $A$ so that $\bar{A} \in B_{p'}$:
    \[ \bar{A}(t) = \frac{t^{p'}}{\log(e+t)^{(1+\epsilon)}} \]
    for some $\epsilon > 0$ to be determined.  Then we have
    \begin{align*}
        A^{-1}(t) \approx{}& \frac{t^{\frac{1}{p}}}{\log(e+t)^{(1+\epsilon){\frac{1}{p'}}}},\\
        U^{-1}(t) \approx{}& \frac{t^{\frac{1}{p}}}{\log(e+t)^{m+\frac{1}{p'} + \frac{\delta}{p}}},\\
        \phi^{-1}(t) \approx{}& \log(e+t).
    \end{align*}
    We want to show $(\phi^{-1}(t))^j U^{-1}(t) \lesssim A^{-1}(t)$.  So
    \begin{align*}
        (\phi^{-1}(t))^j U^{-1}(t) \approx{}& \frac{t^{\frac{1}{p}}}{\log(e+t)^{\frac{1}{p'}+m-j+\frac{\delta}{p}}}\\
        \leq{}& \frac{t^{\frac{1}{p}}}{\log(e+t)^{\frac{1}{p'}+\frac{\delta}{p}}}.
    \end{align*}
    Choosing $\epsilon > 0$ such that $\frac{\epsilon}{p'} \leq \frac{\delta}{p}$ satisfies the conditions of Theorem \ref{Generalized_Young_Holders}.  By a nearly identical argument with $V(t) = t^{p'}\log(e+t)^{(m+\frac{1}{p})p'+\delta}$, we see that $$(\phi^{-1}(t))^{m-j} V^{-1}(t) \lesssim B^{-1}(t),$$ by letting $\epsilon$ be small enough so that $\frac{\epsilon}{p} \leq \frac{\delta}{p'}$.  With the exponential integrability of $BMO$ functions we have
    \begin{align*}
        \left\|\left( \prod_{i \in \tau} (b_i - (b_i)_Q) \right)u^{\frac{1}{p}}\right\|_{A,Q} \leq{}& 
    \left( \prod_{i \in \tau} \|b_i\|_{BMO} \right)\|u^{\frac{1}{p}}\|_{L^p(\log L)^{(m+\frac{1}{p'})p + \delta}}\\
    \intertext{and}
    \left\|\left( \prod_{l \in \tau^c} (b_l - (b_l)_Q) \right)v^{-\frac{1}{p}}\right\|_{B,Q} \leq{}& 
    \left( \prod_{l \in \tau^c} \|b_l\|_{BMO} \right) \|v^{-\frac{1}{p}}\|_{L^{p'}(\log L)^{(m+\frac{1}{p})p' + \delta}}
    \end{align*}
    With this we may use Theorem \ref{Norm_ineq_for_CZO} and have our desired result:
    \begin{align*}
        \|T_{\mathbf{b}}f\|_{L^p(u)} & \\ \lesssim{}&\sum_{\tau\subseteq \{1,\ldots,m\}} \sup_Q\left\|\left( \prod_{i \in \tau} (b_i - (b_i)_Q) \right)u^{\frac{1}{q}}\right\|_{A,Q} \left\|\left( \prod_{l \in \tau^c} (b_l - (b_l)_Q) \right)v^{-\frac{1}{p}}\right\|_{B,Q}\|f\|_{L^p(v)}\\
        \lesssim{}& \sup_Q \left( \prod_{i=1}^m \|b_i\|_{BMO} \right)\|u^{\frac{1}{p}}\|_{L^p(\log L)^{(m+\frac{1}{p'})p + \delta}}\|v^{-\frac{1}{p}}\|_{L^{p'}(\log L)^{(m+\frac{1}{p})p' + \delta}}\|f\|_{L^p(v)}\\
        \leq{}& \left( \prod_{i=1}^m \|b_i\|_{BMO} \right)\|f\|_{L^p(v)}.
    \end{align*}
\end{proof}

\section{Proof of Theorems \ref{main_app_CZO} and \ref{main_app_frac}}
In this section we prove our compactness results, Theorems \ref{main_app_CZO} and \ref{main_app_frac}.  

\begin{proof}[Proof of Theorem  \ref{main_app_CZO}]
    Consider the unit ball in $L^p(v)$,
    \begin{align*}
    B_{L^p(v)}:=\{f\in L^p(v):\|f\|_{L^p(v)}\leq1\}.
    \end{align*}
    We need to show that the set
    $$\mathcal{F}=T_{b}^m(B_{L^p(v)})$$
    satisfies $(a)$-$(c)$ of Lemma \ref{GuoZhao}. We will make some reductions to prove it. From the definition of $CMO(\mathbb{R}^n)$, for any $\epsilon>0$, there is a function $b_\epsilon\in C_c^\infty(\mathbb{R}^n)$ such that
    \begin{align*}
    \|b-b_\epsilon\|_{BMO(\mathbb{R}^n)}<\epsilon.
    \end{align*}
    Observe that
    \begin{align*}
    &\big(b(x)-b(y)\big)^m-\big(b_\epsilon(x)-b_\epsilon(y)\big)^m\\
    &\quad=
    \big((b-b_{\epsilon})(x)-(b-b_{\epsilon})(y)\big)\sum_{i+j=m-1}(b(x)-b(y))^i
    (b_{\epsilon}(x)-b_{\epsilon}(y))^j.
    \end{align*}
    This, together with Theorems \ref{Norm_ineq_for_CZO_BMO} and \ref{Norm_ineq_for_FracOp_BMO}, allows us to deduce that
    \begin{align*}
    &\big\|T_b^mf-T_{b_\epsilon}^mf\big\|_{L^p(\mu)}\\
    &\quad\leq\sum_{i=0}^{m-1}\|T_{\textbf{b}_i}^mf\|_{L^p(\mu)}\\
    &\quad\lesssim\sum_{i=0}^{m-1}\|b-b_\epsilon\|_{BMO}\|b_\epsilon\|_{BMO}^i
    \|b_\epsilon\|_{BMO}^{m-1-i}\lesssim\epsilon\|b\|_{BMO}^{m-1},
    \end{align*}
    where
    \begin{align*}
    \textbf{b}_0:=(b-b_\epsilon,\underbrace{b_\epsilon,\cdots,b_\epsilon}_{m-1}),~\textbf{b}_{m-1}:=
    (b-b_\epsilon,\underbrace{b,\cdots,b}_{m-1})
    \end{align*}
    and
    \begin{align*}
    \textbf{b}_i:=(b-b_\epsilon,\underbrace{b,\cdots,b}_{i},b_\epsilon,\cdots,b_\epsilon),~i\in\tau_{m-2}.
    \end{align*}
    Given that the space of compact operators $\mathcal{K}(L^p(v),L^p(u))$ is a closed subset of the space of bounded operators $\mathcal{B}(L^p(v),L^p(u))$, then $T_{b_\epsilon}^m f \rightarrow T_{b}^mf$ in the operator topology as $b_\epsilon \rightarrow b$ in $BMO$.  Hence, it suffices to show
    $$\mathcal{F}=T_{b}^m(B_{L^p(v)}),~b\in C_c^\infty(\mathbb{R}^n)$$
    satisfies $(a)$-$(c)$ of Lemma \ref{GuoZhao}.
    
    It will also be useful to employ a smooth truncation of our operators, an idea that goes back to Krantz and Li \cite{MR1835563} (see also \cite{CompactCommutatorsWeighted}). Given $\eta>0$, we let $K^\eta(x,y)$ be a smooth truncation of the $K(x,y)$ such that 
    \begin{enumerate}
        \item $K^\eta(x,y)=0$ if $|x-y|\leq \eta$;
        \item $K^\eta(x,y)=K(x,y)$ if $|x-y|> 2\eta$;
        \item $K^\eta$ satisfies the same size and regularity estimates as $K$, namely
        \begin{equation}\label{sizesmooth} |K^\eta(x,y)|\leq \frac{C}{|x-y|^{n}} \quad \text{and} \quad |\nabla K^\eta(x,y)|\leq \frac{C}{|x-y|^{n+1}}.\end{equation}
    \end{enumerate}
    Let $T^{\eta}$ be the operator associated with our truncated kernel $K^\eta$.  Note that
    \begin{equation}\label{truncate_bound}
        |T^{\eta} f(x)| \leq 
        C(Mf(x) + T^\sharp f(x)) 
    \end{equation}
    Thus $T^\eta$ is a bounded operator from $L^p(v)$ to $L^p(u)$.  Let $T_{b}^{m,\eta}$ be the iterated commutators associated with our truncated operator, we have
    \begin{align*}
        |T_{b}^{m,\eta}f(x) - T_{b}^m f(x)| \lesssim{}& \int_{|x-y| \leq 2\eta} |b(x) - b(y)|^m \frac{|f(y)|}{|x-y|^{n}}dy\\
        \lesssim{}& \|\nabla b\|_{L^\infty}^m \sum_{k=0}^\infty (\eta ^{-k})^{m - n} \int_{\eta 2^{-k} < |x-y| \leq \eta 2^{-k+1}} |f(y)|dy\\
        \leq{}& \eta^m \sum_{k=0}^\infty 2^{-km}(\eta 2^{-k})^{-n} \int_{|x-y| \leq \eta 2^{-k+1}} |f(y)|dy\\
        \lesssim{}& \eta^m M f(x).
    \end{align*}
 With this we have
    \[ \|T_{b}^{m,\eta}f - T_{b}^m f\|_{L^p(u)} \lesssim \eta^m \|M f\|_{L^p(u)} \lesssim \eta^m, \]
    which gives us that $T_{b}^{m,\eta} \rightarrow T_{b}^m$ as $\eta \rightarrow 0$ in the operator topology.  Therefore we must show that $\mathcal{F} = T_{b}^{m,\eta}(B_{L^p(v)})$ satisfies (a)-(c) of Lemma \ref{GuoZhao}.
    
    By \eqref{truncate_bound} it is easy to show that $L_{\alpha,b}^{m,\eta}$ is bounded from $L^p(v)$ to $L^p(u)$ and
    \[ \|T_{b}^{m,\eta}\|_{L^p(u)} \leq C\|f\|_{L^p(v)} \leq C, \]
    thus $\mathcal{F}$ satisfies (a) of Lemma \ref{GuoZhao}.  To satisfy (b) from Lemma \ref{GuoZhao} let $$A > \max\{1, 2\sup\{ |y| \,:\, y \in \mathsf{supp }\ b\}\}$$ and let $Q$ be a cube containing $\mathsf{supp}\,b$.  Then for $|x| > A$ and $y \in \mathsf{supp}\,b$ we have $|x-y| > |x|/2$ and
    \begin{align*}
        |T_{b}^{m,\eta} f(x)| \leq{}&
        \int_{\R^n} \frac{|b(x) - b(y)|^m}{|x-y|^{n}}|f(y)|dy \\ \leq{}&
        \frac{C\|b\|_{L^\infty}^m}{|x|^{n}} \int_{\mathsf{supp}\ b} |f(y)|dy\\
        \leq{}& \frac{C\|b\|_{L^\infty}}{|x|^{n}} \|f\|_{L^p(v)}\left( \int_Q v^{-\frac{p'}{p}} \right)^{\frac{1}{p'}}\\
        \leq{}& \frac{C_{b,v}}{|x|^{n}},
    \end{align*}
    where $C_{b,v}$ depends on $b$ and $v$, but not $f$.  To see this, note that
    \[ \left( \int_Q v^{-\frac{p'}{p}} \right)^{\frac{1}{p'}} \leq C|Q|^{\frac{1}{p'}}\|v^{-\frac{1}{p}}\|_{L^{p'}(\log L)^{2p'-1+\delta}}<\infty. \]
    Thus
    \begin{equation}\label{Vanish_at_inf_inequality}
         \int_{|x|>A} |T_{b}^{m,\eta} f(x)|^p u(x)dx \leq C_{b,v}^p   \int_{|x|>A} \frac{u(x)}{|x|^{np}}dx.
    \end{equation}
    We must show that the right-hand side of \eqref{Vanish_at_inf_inequality} is finite.  Noting that $M:L^p(v) \rightarrow L^p(u)$ and
    \[ M (\chi_{[-1,1]^n})(x) \geq \frac{C}{(|x|+1)^{n}}, \]  
    then
    \[ \int_{\R^n} \frac{u(x)}{(|x|+1)^{np}}dx \leq {C} \int_{\R^n} M (\chi_{[-1,1]^n})(x)^p u(x)dx < \infty. \]
    Thus we have
    \[ \lim_{|A| \rightarrow \infty} \left( \int_{|x|>A} |T_{b}^{m,\eta} f(x)|^p u(x)dx \right)^{1/p} = 0 \]
    and so $\mathcal{F}$ satisfies (b) of Lemma \ref{GuoZhao}.
    
    Finally, we check $(c)$ of Lemma \ref{GuoZhao} holds. Take $h\in\mathbb{R}^n$ such that $|h|\leq\frac{\eta}{2}$. It is direct that
    \begin{align}\label{equi decom}
    &T_{b}^{m,\eta}f(x+h)-T_{b}^{m,\eta}f(x)\\
    &\quad=\int_{\mathbb{R}^n}(b(x+h)-b(y))^m(K^\eta(x+h,y)-K^\eta(x,y))f(y)dy\nonumber\\
    &\qquad+\int_{\mathbb{R}^n}[(b(x+h)-b(y))^m-(b(x)-b(y))^m]K^\eta(x,y)f(y)dy\nonumber\\
    &\quad:=Af(x,h)+Bf(x,h)\nonumber.
    \end{align}
    
   The kernels $K^\eta(x+h,y)$ and $K^\eta(x,y)$ both vanish when $|x-y|<\eta/2$ and by \eqref{sizesmooth} we have
    \begin{align*}
    Af(x,h) &\lesssim\|b\|_{L^\infty}^m\int_{|x-y|\geq\eta/2}\frac{|h|}{|x-y|^{n+1}}|f(y)|dy\\
    &\leq \|b\|_{L^\infty}^m\sum_{k=0}^\infty  \int_{2^{k-1}\eta\leq|x-y|<2^{k}\eta}
    \frac{|h|}{|x-y|^{n+1}}|f(y)|dy\\
    &\lesssim\frac{|h|}{\eta}\|b\|_{L^\infty}^m\sum_{k=0}^\infty \frac{1}{2^k}
    \frac{1}{(\eta 2^{k-1})^{n-\alpha}}\int_{|x-y|\leq\eta 2^k}|f(y)|dy\\
    &\lesssim\frac{|h|}{\eta}\|b\|_{L^\infty}^m Mf(x).
    \end{align*}
    This gives us
    \begin{align}\label{A}
    \sup_{f\in B_{L^p(v)}}\|Af\|_{L^p(u)}\leq C_{b,\delta}|h|\|M f\|_{L^p(u)}\leq C_{b,\delta}|h|,
    \end{align}
    which will go to zero uniformly as $h\rightarrow0$.
    
    For $Bf$ we use the difference of powers formula to write:
    \begin{align*}
 \lefteqn{ (b(x+h)-b(y))^m-(b(x)-b(y))^m}\\
 &\hspace{3cm} =(b(x+h)-b(x)+b(x)-b(y))^m-(b(x)-b(y))^m\\ 
  &\hspace{3cm} =\sum_{k=1}^{m}{m \choose k} (b(x+h) - b(x))^k (b(x) - b(y))^{m-k}\\
  &\hspace{3cm}=\sum_{k=1}^{m}{m \choose k} (b(x+h) - b(x))^k \sum_{j=0}^{m-k}{m-k\choose j}b(x)^j b(y)^{m-k-j}.\\
      \end{align*}
We now break of the regions of integration for $B$ as follows 
  \[ Bf = \int_{|x-y| > 2\eta} + \int_{\eta \leq |x-y| < 2\eta} =: B_1 f + B_2 f \]
 We have the following estimates for $B_1 f$:
    \begin{align*}
        |B_1f(x,h)| &\leq \sum_{k=1}^{m}{m \choose k} |b(x+h) - b(x)|^k \sum_{j=0}^{m-k}{m-k\choose j}|b(x)|^j \\
        &\qquad \times\left|\int_{|x-y|>2\eta} K(x,y) b(y)^{m-k-j}f(y)\,dy\right|\\
        &\leq |h|\|\nabla b\|_{L^\infty}\sum_{k=1}^m{m \choose k}|b(x+h)-b(x)|^{k-1}\\ 
        &\qquad \times\sum_{j=0}^{m-k}{m-k\choose j}|b(x)|^jT^\sharp(b^{m-k-j}f)(x) \\
         &\leq C_b|h|\sum_{k=1}^m\sum_{j=0}^{m-k}T^\sharp(b^{m-k-j}f)(x). 
      \end{align*}
 Hence we have
 \begin{multline*} \|B_1f\|_{L^p(u)}\leq C_b|h|\sum_{k=1}^m\sum_{j=0}^{m-k}\|T^\sharp(b^{m-k-j}f)\|_{L^p(u)} \\ 
 \leq C_b|h|\sum_{k=1}^m\sum_{j=0}^{m-k}\|b^{m-k-j}f\|_{L^p(v)}\leq C_b|h| \|f\|_{L^p(v)}.\end{multline*}   
For $B_2 f$ we have
    \begin{align*}
        |B_2 f(x,h)| \lesssim{}& |h|\|\nabla b \|_{L^\infty}\|b\|_{L^\infty}^{m-1}\int_{\eta < |x-y| \leq 2\eta} \frac{|f(y)|}{|x-y|^{n}}dy\\
        \lesssim{}& C_b |h| \frac{1}{(2\eta)^{n}} \int_{|x-y| \leq 2\eta} |f(y)|dy\\
        \leq{}& C_b |h| M f(x)
        \intertext{which gives us}
        \|B_2 f\|_{L^q(u)} \leq{}& C_b|h|\|f\|_{L^p(v)}.
    \end{align*}
    This completes the equicontinuity argument and so by Lemma \ref{GuoZhao} proves compactness.
    
\end{proof}

\begin{proof}[Proof of Theorem \ref{main_app_frac}]
   The proof for $(I_\al)^m_b$ is similar to that of Theorem \ref{main_app_CZO} and we only sketch the details. However, there are some simplifications that we point out.
 Again, it suffices to show that 
    $$\mathcal{F}=(I_\al)_b^m(B_{L^p(v)})\subseteq L^q(u)$$
    satisfies (a)-(c) of Lemma \ref{GuoZhao}. By the same reductions we may  assume $b\in C_c^\infty(\R^n)$ and the truncated operator for $\eta>0$
    $$I^\eta_\al f(x)=\int_{|x-y|>\eta}K_\al^\eta(x,y)f(y)\,dy$$
    where 
    $$K_\al^\eta(x,y)=\left\{\begin{array}{cc}|x-y|^{\al-n}, & |x-y|\geq 2\eta \\ 0, & |x-y|\leq \eta\end{array}\right.$$
and 
$$|K_\al^\eta(x,y)|\leq \frac{1}{|x-y|^{n-\al}} \quad \text{and} \quad |\nabla K_\al^\eta(x,y)|\leq \frac{C}{|x-y|^{n-\al+1}}.$$
    Notice that 
    $$|I^\eta_\al f(x)|\leq I_\al(|f|)(x)$$
    and thus is bounded $I_\al^\eta:L^p(v)\ra L^q(u)$ when the weights satisfy condition \eqref{Iterated_Riesz_Commutator_Bump_Conditions}.  Moreover consider the truncated commutator
    $$(I^\eta_\al)_b^mf(x)=\int_{\R^n}{(b(x)-b(y))^m}K_\al^\eta(x,y)f(y)\,dy.$$ 
    Then we have     
    \begin{align*}
        |(I^\eta_\al)_b^mf(x) - (I_\al)_b^mf(x)| \lesssim \eta^m M_\alpha f(x),
    \end{align*}
    and this implies that that $(I^\eta_\al)_b^m$ is a bounded operator from $L^p(v)$ to $L^q(u)$ and
    \[ \|(I^\eta_\al)_b^mf - (I_\al)_b^mf\|_{L^q(u)} \lesssim \eta^m \|M_\alpha f\|_{L^p(u)} \lesssim \eta^m \]
    which gives us that $(I^\eta_\al)_b^m \rightarrow (I^\eta_\al)_b^m$ as $\eta \rightarrow 0$ in the operator topology.  Therefore we must show that $\mathcal{F} = (I^\eta_\al)_b^m(B_{L^p(v)})$ satisfies (a)-(c) of Lemma \ref{GuoZhao}. Since $(I^\eta_\al)_b^m$ is a bounded operator we have that $\mathcal{F}$ satisfies (a) of Lemma \ref{GuoZhao}.  The proof of (b) from Lemma \ref{GuoZhao} is completely analogous: for $A$ and let $Q$ be a cube containing $\mathsf{supp}\,b$ we have
$$ |(I^\eta_\alpha)_b^{m} f(x) \leq{} \frac{C_{b,v}}{|x|^{n-\alpha}}$$
and
    \begin{equation*}
         \int_{|x|>A} |(I^\eta_\alpha)_b^{m} f(x)|^q u(x)dx \leq C_{b,v}^q   \int_{|x|>A} \frac{u(x)}{|x|^{(n-\alpha)q}}dx.
    \end{equation*}
 Noting that $M_\alpha:L^p(v) \rightarrow L^q(u)$ and
    \[ M_\alpha (\chi_{[-1,1]^n})(x) \geq \frac{C}{(|x|+1)^{n-\alpha}}, \]  
   yields
    \[ \lim_{A \rightarrow \infty} \left( \int_{|x|>A} |(I^\eta_\alpha)_b^{m}  f(x)|^q u(x)dx \right)^{1/q} = 0 \]
    and so $\mathcal{F}$ satisfies (b) of Lemma \ref{GuoZhao}.
    
    To check (c) of Lemma \ref{GuoZhao} we take a similar approach. For $h\in\mathbb{R}^n$ such that $|h|\leq\frac{\eta}{2}$, we have
    \begin{align*}
    &(I^\eta_\alpha)_b^m f(x+h)-(I^\eta_\alpha)_b^m f(x)\\
    &\quad=\int_{\mathbb{R}^n}(b(x+h)-b(y))^m(K_\alpha^\eta(x+h,y)-K_\alpha^\eta(x,y))f(y)dy\\
    &\qquad+\int_{\mathbb{R}^n}[(b(x+h)-b(y))^m-(b(x)-b(y))^m]K_\alpha^\eta(x,y)f(y)dy\\
    &\quad=:Af(x,h)+Bf(x,h)\nonumber.
    \end{align*}
    
    One can see that $K_\alpha^\eta(x+h,y)$ and $K_\alpha^\eta(x,y)$ both vanish when $|x-y|<\frac{\eta}{2}$ and by the smoothness condition on $K_\al^\eta$ we have
    \begin{align*}
    Af(x,h) &\lesssim\|b\|_{L^\infty}^m\int_{|x-y|\geq\frac\eta2}\frac{|h|}{|x-y|^{n-\alpha+1}}|f(y)|dy\\
    &\leq \|b\|_{L^\infty}^m\sum_{k=0}^\infty  \int_{2^{k-1}\eta\leq|x-y|<2^{k}\eta}
    \frac{|h|}{|x-y|^{n-\alpha +1}}|f(y)|dy\\
    &\lesssim\frac{|h|}{\eta}\|b\|_{L^\infty}^m\sum_{k=0}^\infty \frac{1}{2^k}
    \frac{1}{(\eta 2^{k-1})^{n-\alpha}}\int_{|x-y|\leq 2^{k}\eta}|f(y)|dy\\
    &\lesssim\frac{|h|}{\eta}\|b\|_{L^\infty}^m M_\alpha f(x).
    \end{align*}
    This gives us
    \begin{align*}
    \sup_{f\in B_{L^p(v)}}\|Af\|_{L^q(u)}\leq C|h|\|M_\al f\|_{L^p(u)}\leq C|h|,
    \end{align*}
    which will go to zero uniformly as $h\rightarrow0$.
    
    For $Bf$ we use a slightly different power formula to write:
    \begin{flalign*}
        (b(x+h)-b(y))^m-&(b(x)-b(y))^m\\ =& (b(x+h) - b(x))  \left( \sum_{k=0}^{m-1} (b(x+h) - b(y))^k (b(x) - b(y))^{m-1-k} \right).
    \end{flalign*}
We now break of the regions of integration for $B$ as follows 
  \[ Bf = \int_{|x-y| > 2\eta} + \int_{\eta \leq |x-y| \leq 2\eta} := B_1 f + B_2 f \]
 We have the following estimates for $B_1 f$:

$$ |B_1f(x,h)| \lesssim{} |h|\|\nabla b \|_{L^\infty}\|b\|_{L^\infty}^{m-1}\int_{|x-y| > 2\eta} \frac{|f(y)|}{|x-y|^{n-\alpha}}dy\leq C_b|h|I_\al f(x),$$
which gives
$$        \|B_1 f \|_{L^q(u)} \lesssim{} |h|\|f\|_{L^p(v)}.$$
For $B_2 f$ we have
    \begin{align*}
        |B_2 f(x,h)| \lesssim{}& |h|\|\nabla b \|_{L^\infty}\|b\|_{L^\infty}^{m-1}\int_{\eta < |x-y| \leq 2\eta} \frac{|f(y)|}{|x-y|^{n-\alpha}}dy\\
        \lesssim{}& C_b |h| \frac{1}{(2\eta)^{n-\alpha}} \int_{|x-y| \leq 2\eta} |f(y)|dy\\
        \leq{}& C_b |h| M_\alpha f(x)
            \end{align*}
which yields,
$$        \|B_2 f\|_{L^q(u)} \leq{} C|h|\|f\|_{L^p(v)}.$$
    This completes the equicontinuity argument and so by Lemma \ref{GuoZhao} proves compactness.

\end{proof}

\bibliographystyle{plain}
\bibliography{biblio}

\end{document}